\documentclass[framed]{customarticle}

\title{Stochastic Nonconvex Bilevel Optimization: \\ Improved Rates Without Rare-Visit Assumption}
\author{
Daniel Cortild$^{1}$ \orcidCortild, Mathias Staudigl$^2$ \orcidStaudigl, Juan Peypouquet$^{3}$ \orcidPeypouquet, Coralia Cartis$^1$ \orcidCartis \\ \\
$^{1}$Mathematical Institute, University of Oxford, United Kingdom \\
$^{2}$Department of Business Analytics and Decision Sciences, \\ University of Economics and Business, Vienna, Austria \\
$^{3}$Bernouilli Institute, University of Groningen, The Netherlands}
\date{Last Compiled: \today}

\usepackage{lipsum}
\usepackage{threeparttable}

\newcommand{\nablafk}{\nabla f_k}
\newcommand{\nablagk}{\nabla g_k}

\begin{document}

\maketitle

\begin{abstract}
\noindent We investigate stochastic simple bilevel optimization with smooth and possibly nonconvex upper- and lower-level objectives. Existing stochastic extensions of dynamic barrier gradient descent (DBGD) either obtain fast convergence under an unverifiable trajectory-dependent ``rare-visit'' assumption, or remove this assumption at a substantially higher oracle cost. We show that a simple denominator-only regularization of the DBGD multiplier eliminates the need for such an assumption while preserving fast convergence rates. Specifically, our method achieves $(\varepsilon, \varepsilon)$-stationarity in $\mathcal O(\varepsilon^{-2})$ iterations using $\mathcal O(\varepsilon^{-4})$ upper-level and $\mathcal O(\varepsilon^{-7})$ lower-level stochastic gradients, which improves upon the best assumption-free complexities. We additionally derive anytime parameter schedules. 
\end{abstract}

\section{Introduction}

In this paper, we consider a \textit{simple stochastic nonconvex bilevel program}, given by
\begin{equation}\label{eq:SSNBP}\tag{P}
    \min_{x\in \H} \left\{f(x) \colon \quad x\in \argmin_{y\in \H} g(y)\right\},
\end{equation}
where $f=\Exp{f_\xi}, g=\Exp{g_\zeta}\colon \H\to \R$ are defined over a real Hilbert space $\H$ and are smooth and not necessarily convex, and only accessible through a stochastic oracle call. Here, the term \textit{simple} refers to the fact that $f$ does not depend on the lower-level variable $y$ and $g$ does not depend on the upper-level variable $x$. In this context, the lower-level function $g$ serves as a primary requirement, whereas the upper-level objective $f$ selects amongst lower-level solutions. This problem is also referred to as a \textit{nonconvex hierarchical program}. Such a formulation has many applications, such as in hyperparameter training \citep{franceschi_bilevel_2018} or in machine unlearning \citep{reisizadeh_blur_2025}, to name a few.

Even in a deterministic setting, nonconvex bilevel programs are hard. The lower-level constraint is global, and may have a nonconvex or even disconnected solution set. Therefore, classical ideas from convex bilevel programming, such as Tikhonov regularizations \cite{kaushik_method_2021,samadi_improved_2025,alves_inertial_2025,cortild_regularization_2025} or forming outer approxiations of the lower-level solution set \cite{jiang_conditional_2023,cao_projectionfree_2023,cao_accelerated_2024} do not translate.

As solving Problem \eqref{eq:SSNBP} to global optimality is intractable, we merely solve it to $(\varepsilon_f, \varepsilon_g)$-stationarity, which was introduced in \cite{cao_complexity_2025} and will be rigorously defined in Section \ref{sec:setup}. It is such that $\nabla g(x)$ is small, so that the point is nearly stationary for the lower-level, and such that there is no evident first-order direction that improves the upper-level objective. 

To achieve this notion of stationarity, Cao et al. \cite{cao_complexity_2025} make use of Dynamic Barrier Gradient Descent (DBGD), originally proposed in \cite{gong_automatic_2021}, which combines upper- and lower-level descent through an adaptive dual multiplier. Extending this approach to the present stochastic setting is nontrivial, as directly replacing the gradient by a stochastic estimator may cause singularities within the multiplier definition, since the lower-level gradient estimator may vanish away from stationary points.

Recently, Ahmadi et al. \cite{ahmadi_stochastic_2026} address this issue by perturbing both the numerator and the denominator of the dual multiplier. The resulting method achieves fast rates under a trajectory-dependent ``rare-visit'' assumption, which is unverifiable \textit{a priori}, and remove this condition through a penality-regularization, which leads to substantially slower rates. We instead perturb only the denominator, giving rise to
\[
    \underbrace{\frac{[\beta(\|\nabla g\|^2+\rho)-\langle \nabla f, \nabla g\rangle]_+}{\|\nabla g\|^2+\rho}}_{\text{Ahmadi et al. \cite{ahmadi_stochastic_2026}}}\quad\text{v.s.}\quad \underbrace{\frac{[\beta\|\nabla g\|^2-\langle \nabla f, \nabla g\rangle]_+}{\|\nabla g\|^2+\rho}}_{\text{Ours}}.
\]
Despite the multipliers differing only by an additive constant in the numerator, this modification allows us to establish faster convergence rates under standard, verifiable assumptions.
We provide a more thorough comparison of these works in Table \ref{tab:comparison}.

\begin{table}[H]
    \centering
    \small
    \begin{threeparttable}
        \caption{Comparison of iteration count, and upper- and lower-level
        (stochastic) oracle calls to reach a $(\varepsilon,\varepsilon)$-stationary
        point, and the additionally required assumptions.}
        \label{tab:comparison}

        \begin{tabular}{|l|cccc|l|}
            \hline
            \textbf{Algorithm}
            & \textbf{Setting}
            & \textbf{Iterations}
            & \textbf{Upper}
            & \textbf{Lower}
            & \textbf{Extra Assumption?}
            \\ \hline

            \textbf{DBGD} \cite{cao_complexity_2025}
            & Det
            & $\mathcal O(\varepsilon^{-2})$
            & $\mathcal O(\varepsilon^{-2})$
            & $\mathcal O(\varepsilon^{-2})$
            & N/A
            \\ \hline

            \textbf{SDBPG} \cite{ahmadi_stochastic_2026}
            & Sto
            & $\mathcal O(\varepsilon^{-2})$
            & $\mathcal O(\varepsilon^{-2/\varsigma})$
            & $\mathcal O(\varepsilon^{-3/\varsigma})$
            & Yes\tnote{1,2}
            \\

            \textbf{PR-SDBPG} \cite{ahmadi_stochastic_2026}
            & Sto
            & $\mathcal O(\varepsilon^{-4})$
            & $\mathcal O(\varepsilon^{-8})$
            & $\mathcal O(\varepsilon^{-10})$
            & Yes, but verifiable\tnote{2}
            \\

            \textbf{VR-PR-SDBPG} \cite{ahmadi_stochastic_2026}
            & Sto
            & $\mathcal O(\varepsilon^{-4})$
            & $\mathcal O(\varepsilon^{-6})$
            & $\mathcal O(\varepsilon^{-8})$
            & Yes, but verifiable\tnote{2,3}
            \\ \hline

            \textbf{SDBGD} (This work)
            & Sto
            & $\mathcal O(\varepsilon^{-2})$
            & $\mathcal O(\varepsilon^{-4})$
            & $\mathcal O(\varepsilon^{-7})$
            & N/A
            \\ \hline
        \end{tabular}

        \begin{tablenotes}[flushleft]
            \footnotesize
            \item[1] ``Rate-Visit'' Assumption $\varsigma\in (0, 1/2]$: unverifiable \textit{a priori} as it involves the trajectory of the algorithm.
            \item[3] Mean-squared smoothness of the individual realization $f_\xi$ and $g_\zeta$.
        \end{tablenotes}
    \end{threeparttable}
\end{table}

\vspace{-1.5em}

\textbf{Contributions.} We propose a new algorithm, SDBGD, to solve Problem \eqref{eq:SSNBP}, converging to an $(\varepsilon, \varepsilon)$-stationary point in $\mathcal O(\varepsilon^{-2})$ iterations, requiring a total of $\mathcal O(\varepsilon^{-4})$ stochastic oracle calls to $f$ and $\mathcal O(\varepsilon^{-7})$ stochastic oracle calls to $g$. Importantly, we do not rely on an unverifiable assumption, such as the ``rare-visit'' assumption. We additionally provide anytime convergence rates, not requiring knowledge of the final horizon. We finally compare our method numerically to the ones proposed by Ahmadi et al. \cite{ahmadi_stochastic_2026} on a simple example.
\section{Problem Setup and Stationarity}\label{sec:setup}

Let $\H$ be a real Hilbert space with associated inner product $\langle \cdot, \cdot\rangle$ and induced norm $\|\cdot\|$. We consider the stochastic simple bilevel problem
\begin{equation}\tag{P}
    \min_{x\in \H} \left\{f(x) \colon \quad x\in \argmin g\right\},
\end{equation}
where, for all $x\in \H$, $f(x)=\ExpD{\xi}{f_\xi(x)}$ and $g(x)=\ExpD{\zeta}{g_\zeta(x)}$, 
with $f_\xi, g_\zeta\colon \H\to \R$ stochastic realizations of $f$ and $g$ respectively. We assume the following.
\begin{assumption}\label{ass:problem}
The expectations defining $f$ and $g$ are well-defined and differentiable, with, for all $x\in \H$, $\nabla f(x)=\ExpD{\xi}{\nabla f_\xi(x)}$ and $\nabla g(x)=\ExpD{\zeta}{\nabla g_\zeta(x)}$. 

We further assume that $\argmin g\neq \emptyset$ and that $\inf f>-\infty$. 
\end{assumption}

We additionally assume the following.

\begin{assumption}\label{ass:main}
    There exist constants $L_f, L_g, G_f>0, \sigma_f^2, \sigma_g^2\ge 0$ such that
    \begin{enumerate}
        \item $f$ and $g$ are respectively $L_f$- and $L_g$-smooth;
        \item the upper-level gradient is uniformly bounded, namely $\|\nabla f(x)\|\le G_f$ for all $x\in \H$;
        \item the stochastic gradients have uniformly bounded variance, namely 
        \[
            \ExpD{\xi}{\|\nabla f_\xi(x)-\nabla f(x)\|^2}\le \sigma_f^2\quad\text{and}\quad \ExpD{\zeta}{\|\nabla g_\zeta(x)-\nabla g(x)\|^2}\le \sigma_g^2\quad \forall x\in \H.
        \]
    \end{enumerate}
\end{assumption}

Uniformly bounded variance is a standard assumption in stochastic optimization \cite{robbins_convergence_1971,nemirovski_robust_2009}, although recent trends aim to relax it under weaker growth or variance-at-solution conditions \cite{bach_nonasymptotic_2011,needell_stochastic_2016,khaled_better_2023,alacaoglu_weaker_2025,cortild_biasoptimal_2025,garrigos_lastiterate_2025,cortild_stochastic_2026}, at the cost of a more complex error handling. We retain the bounded-variance assumption here in order to simplify the results and isolate the difficulties arising from the stochasticity itself.

Assumption \ref{ass:main} does not require the lower-level gradient $\nabla g$ to be uniformly bounded, nor does it require the individual realizations $f_\xi$ and $g_\zeta$ to be smooth. This is in contrast with \cite{ahmadi_stochastic_2026}, where boundedness of $\nabla g$ is required, and smoothness of the stochastic realizations is used for their variance-reduced variant.

Since Problem \eqref{eq:SSNBP} is nonconvex, we do not aim to solve it globally and instead introduce a notion of stationarity. We follow the convention of \cite{cao_complexity_2025} and measure stationarity in terms of squared gradient norms. This choice is equivalent to measuring it in terms of norms and rescaling $(\varepsilon_f, \varepsilon_g)$, and is convenient as our analysis naturally controls the squared quantities.

\begin{definition}[$(\varepsilon_f, \varepsilon_g)$-stationarity \cite{cao_complexity_2025}]\label{def:stationarity}
    A point $x\in \H$ is \textit{$(\varepsilon_f, \varepsilon_g)$-stationary} for Problem \eqref{eq:SSNBP} if there exists $\lambda\ge 0$ satisfying 
    \[
        \|\nabla f(x)+\lambda \nabla g(x)\|^2\le \varepsilon_f \quad \text{and}\quad \|\nabla g(x)\|^2\le \varepsilon_g.
    \]
\end{definition}

The second condition requires $x$ to be near-stationary for the lower-level objective. We emphasize that, since $g$ is nonconvex, lower-level stationarity does in general not mean proximity to $\argmin g$. Definition \ref{def:stationarity} should therefore be understood as a first-order approximation for solving Problem \eqref{eq:SSNBP}. The first condition describes first-order compatibility between the upper- and lower-level objectives. In fact, up to error $\sqrt{\varepsilon_f}$, the component of $\nabla f(x)$ orthogonal to $\nabla g(x)$ vanishes, while its parallel component either vanishes or points in the opposite direction to $\nabla g(x)$. As such, there is no first-order direction that can substantially decrease $f$ without increasing $g$.
\section{Stochastic Dynamic Barrier Gradient Descent}\label{sec:alg}

The deterministic Dynamic Barrier Gradient Descent (DBGD) method of \cite{cao_complexity_2025} updates $x_k$ with step-size $\eta_k>0$ along the direction
\[
    d_k^{\text{DBGD}}=\nabla f(x_k)+\lambda_k^{\text{DBGD}} \nabla g(x_k), \quad \lambda_k^{\text{DBGD}}=\frac{[\beta_k\|\nabla g(x_k)\|^2-\langle \nabla f(x_k), \nabla g(x_k)\rangle]_+}{\|\nabla g(x_k)\|^2}.
\]
The choice of multiplier $\lambda_k^{\text{DBGD}}$ is so that $\langle \nabla g(x_k), d_k^{\text{DBGD}}\rangle\ge \beta_k\|\nabla g(x_k)\|^2$, enforcing descent in the lower-level objective while retaining as much as possible in the upper-level. This multiplier may be viewed as the Lagrange multiplier associated with the projection problem $\min\{\frac{1}{2}\|\nabla f(x_{k})-d\|^{2}\colon \; \beta_{k}\|\nabla g(x_{k})\|^{2}\leq \langle d,\nabla g(x_{k})\rangle\}$, namely
\begin{equation}\label{eq:projection_mult}
    \lambda_{k}^{\text{DBGD}}=\arg\max_{\lambda\geq 0}\min_{d}\left\{\frac{1}{2}\|\nabla f(x_{k})-d\|^{2}+\lambda(\beta_{k} \|\nabla g(x_{k})\|^{2}-\langle d,\nabla g(x_{k})\rangle)\right\}.
\end{equation}
The term \textit{dynamic barrier} refers to the lower bound on $\langle d,\nabla g(x_{k})\rangle$ enforced in the projection dynamics, which provides a trade-off between loss minimization and constraint satisfaction, and is not directly related to an explicit barrier function.

In the stochastic setting, we no longer access the full gradients $\nabla g(x_k)$ and $\nabla f(x_k)$, but only unbiased stochastic mini-batch estimators of batch-sizes $B_{f, k}$ and $B_{g, k}$ respectively, given by
\begin{equation}\label{eq:minibatch}
    \nablafk = \frac{1}{B_{f, k}}\sum_{i=1}^{B_{f, k}}\nabla f_{\xi_{k, i}}(x_k)\quad \text{and}\quad \nablagk = \frac{1}{B_{g, k}}\sum_{j=1}^{B_{g, k}}\nabla g_{\zeta_{k, j}}(x_k).
\end{equation}
Directly substituting these quantities into the multiplier definition may be undefined, as $\|\nabla g_k\|$ may be small, or even zero, far from lower-level stationary points. 

We instead regularize the denominator with a parameter $\rho_k>0$, thus dividing by $\|\nablagk\|^2+\rho_k$ rather than $\|\nablagk\|^2$. This prevents the singularity in the multiplier definition, while recovering the deterministic definition as $\rho_k\to 0$. This yields the multiplier definition
\begin{equation}\label{eq:mult}
    \hat \lambda_k=\frac{[\beta_k\|\nablagk\|^2-\langle \nablafk, \nablagk\rangle]_+}{\|\nablagk\|^2+\rho_k}.
\end{equation}
Importantly, unlike the multiplier definition in \cite{ahmadi_stochastic_2026}, we do not add a corresponding $\rho_k\beta_k$ term to the numerator. Although the definitions are very similar, they lead to fundamentally different convergence properties, and require distinct arguments.

We introduce $\lambda_k$ and $d_k$ as quantities for our analysis, but they do not need to be computed:
\[
    \lambda_k=\frac{[\beta_k\|\nabla g(x_k)\|^2-\langle \nabla f(x_k), \nabla g(x_k)\rangle]_+}{\|\nabla g(x_k)\|^2+\rho_k}, \quad d_k=\nabla f(x_k)+\lambda_k \nabla g(x_k).
\]
In the following algorithm, $\eta_k>0$ is the step-size, $\beta_k>0$ is the barrier parameter, $\rho_k>0$ is the regularization parameter, and $B_{f,k}\ge1$ and $B_{g,k}\ge1$ represent the batch-sizes for the upper- and lower-level gradient estimators, respectively.

\begin{algorithm}[H]
\caption{Stochastic Dynamic Barrier Gradient Descent (SDBGD)}\label{alg:SDBGD}
\begin{algorithmic}
\For{$k=0, \ldots, K-1$}
\State Draw independent $\xi_{k, i}$ and $\zeta_{k, j}$ for $i=1, \ldots, B_{f, k}$ and $j=1, \ldots, B_{g, k}$. 
\State Define the independent mini-batches as in \eqref{eq:minibatch}. 
\State Compute $\hat\lambda_k$ according to \eqref{eq:mult}. 
\State Compute $\hat d_k=\nablafk+\hat\lambda_k\nablagk$. 
\State Update $x_{k+1}=x_k-\eta_k \hat d_k$.
\EndFor
\end{algorithmic}
\end{algorithm}

We make use of the random iterate $x_{N_K}$, where $N_K$ is drawn at random according to 
\begin{equation}\label{eq:def_P}
    \mathbb{P}(N_K=k)=\frac{\eta_k\beta_k}{\sum_{k'=0}^{K-1}\eta_{k'}\beta_{k'}}\quad \text{for $k\in \{0, \ldots, K-1\}$}.
\end{equation}
Balancing the stochastic error with the regularization parameter $\rho_k$ yields the following horizon-dependent and anytime rates. The proof is deferred to Appendix \ref{app:proof_coro_rates}, and all proportionality constants are independent of $K$ (or $k$).

\begin{corollary}[Convergence Rates]\label{coro:rates}
    Let Assumptions \ref{ass:problem} and \ref{ass:main} hold, let $(x_k)$ be the iterates generated by Algorithm \ref{alg:SDBGD}, and let $N_K$ be chosen at random according to \eqref{eq:def_P}.
    \begin{enumerate}
        \item \textbf{(Horizon-Dependent)} Given a horizon $K\ge 1$, let $a\in (0, 1/3)$, and select 
        \[
            \eta_k\equiv \eta \propto K^{-a}\in (0, (2L_f+4L_g)^{-1}), ~~~ \beta_k\equiv \beta \propto \eta^{-3}K^{-1}\in (0, 1), ~~~ \rho_k\equiv \rho \propto \eta^{-2}K^{-2},
        \]
        \[
            B_{f,k}\equiv B_f\propto \eta^4K^2, \quad \text{and}\quad B_{g,k}\equiv B_g\propto \eta^6K^4.
        \]
        Then it holds that
        \[
            \Exp{\|d_{N_K}\|^2}\le \mathcal O(K^{-1+2a})\quad \text{and}\quad \Exp{\|\nabla g(x_{N_K})\|^2}\le \mathcal O(K^{-2a}).
        \]
        \item \textbf{(Anytime)} Let $a\in (0, 1/3)$, and select 
        \[
            \eta_k\propto (k+1)^{-a}\in (0, (2L_f+4L_g)^{-1}), \quad \beta_k \propto \eta_k^{-3}(k+1)^{-1}\in (0,1), \quad \rho_k \propto \eta_k^{-2}(k+1)^{-2},
        \]
        \[
            B_{f,k}\propto \eta_k^4(k+1)^2, \quad \text{and}\quad B_{g,k}\propto \max(\eta_k^{-2}, \eta_k^6(k+1)^4).
        \]
        Then it holds that, for any $K\ge 2$, 
        {\small\[
            \Exp{\|d_{N_K}\|^2}\le \begin{dcases}
                \mathcal O(K^{-2a})\quad &\text{if $0<a<1/4$}, \\
                \mathcal O(K^{2a-1}\log(K))&\text{if $1/4\le a<1/3$},
            \end{dcases}
        \]
        and
        \[
            \Exp{\|\nabla g(x_{N_K})\|^2}\le \mathcal O\left(K^{-2a}\log(K)\right).
        \]}
    \end{enumerate}
\end{corollary}

The parameter $a\in (0, 1/3)$ controls the tradeoff between the two stationary residuals. The balanced choice $a=1/4$ yields a rate of ${\mathcal O}(K^{-1/2})$ or $\tilde{\mathcal O}(K^{-1/2})$ for both.

\begin{corollary}[Oracle Complexity]\label{coro:p1}
    Under the horizon-dependent schedule of Corollary \ref{coro:rates} with $a=1/4$, we reach a $(\varepsilon, \varepsilon)$-stationary point in $\mathcal O(\varepsilon^{-2})$ iterations, resulting in $\mathcal O(\varepsilon^{-4})$ upper-level and $\mathcal O(\varepsilon^{-7})$ lower-level stochastic oracle calls.
\end{corollary}

\section{Numerical Illustration}\label{sec:num}

We consider the following toy problem: 
\[
    \min_{x\in \R^2}\left\{f(x)=\sqrt{1+\|x-c\|^2}\colon \quad x\in \argmin_{y=(y_1,y_2)\in \R^2}g(x)=2-\cos(y_1)-\exp(-y_2^2/2)\right\},
\]
using stochastic gradients $\nabla f_\xi =\nabla f+\xi$ and $\nabla g_\zeta=\nabla g + \zeta$ with Gaussian noise. We compare SDBGD with SDBPG, PR-SDBPG and VR-PR-SDBPG \cite{ahmadi_stochastic_2026} under a common oracle budget of $10^8$, using anytime schedules for all methods. The details are provided in Appendix \ref{app:XP}.

In Figure \ref{fig:comparison}, we plot the upper- and lower-level stationary measures for all four methods. While our method does not empirically improve upon the previously proposed ones, its theoretical guarantees are stronger and does not require any unverifiable assumption.

\begin{figure}[H]
    \centering
    \includegraphics[width=1\linewidth]{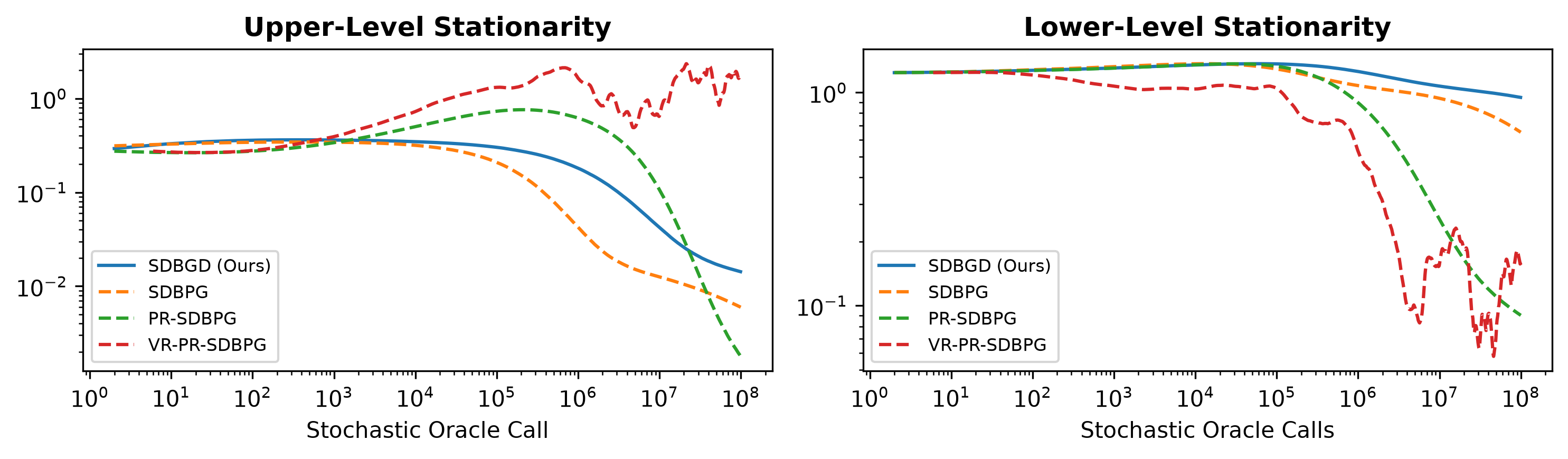}
    \caption{Comparison of stationarity measures in terms of stochastic oracle calls.}
    \label{fig:comparison}
\end{figure}
\section{Conclusion}

We introduced SDBGD, a denominator-regularized stochastic variant of DBGD \cite{cao_complexity_2025} for stochastic simple nonconvex bilevel programming. Under standard assumptions, SDBGD achieves $\mathcal O(\varepsilon^{-2})$ iteration complexity, using $\mathcal O(\varepsilon^{-4})$ upper-level and $\mathcal O(\varepsilon^{-7})$ lower-level stochastic oracle calls. These rates do not require an unverifiable trajectory-dependent ``rare-visit'' assumption as in \cite{ahmadi_stochastic_2026}, and improve upon the assumption-free guarantees of \cite{ahmadi_stochastic_2026}. We additionally provided an anytime parameter schedule and a simple numerical illustration. Overall, our results show that minor differences in the multiplier definition can substantially affect the theoretical guarantees.

\textbf{Open Question.} The stochastic variants of DBGD proposed here and in \cite{ahmadi_stochastic_2026} differ only slightly in their multiplier definition, yet yield considerably different guarantees. This suggests directly substituting the gradient for a stochastic estimator thereof may not be the natural stochastic analogue of DBGD. Since deterministic DBGD admits a projected dynamical system interpretation \eqref{eq:projection_mult}, an interesting direction might be to consider a stochastic projected dynamical system to derive the stochastic multiplier. We leave the following for future work:
\begin{quote}
    \textit{What is the appropriate stochastic analogue of deterministic multiplier-based bilevel methods?}
\end{quote}

\textbf{Acknowledgments.} Daniel Cortild acknowledges the support of the Clarendon Funds Scholarship. 
Mathias Staudigl's research is supported by the Deutsche Forschungsgemeinschaft (DFG) -- Projektnummer 556222748 ``non-stationary hierarchical minimization’'.
Juan Peypouquet benefited from the support of the FMJH Program Gaspard Monge for optimization and operations research and their interactions with data science.
Coralia Cartis' work was supported by the Hong Kong Innovation and Technology Commission (InnoHK Project CIMDA) and by the EPSRC grant EP/Y028872/1, Mathematical Foundations of Intelligence: An “Erlangen Programme” for AI. 

\setlength{\bibsep}{2pt}
\bibliographystyle{abbrv}
\bibliography{references,additional_references}

\newpage
\appendix

\section{Proofs}\label{app:proofs}

Throughout Appendix \ref{app:proofs}, we let $(x_k)$ be the iterates generated by Algorithm \ref{alg:SDBGD}, and we assume Assumptions \ref{ass:problem} and \ref{ass:main} hold. We define the filtration $\mathcal F_k=\sigma(x_0, \{\xi_{\ell, i}, \zeta_{\ell, j}\}_{\ell < k})$ so that $x_k$ is $\mathcal F_k$-measurable, and denote by $\ExpD{k}{\cdot}=\Exp{\cdot|\mathcal F_k}$ the conditional expectation. We let $\Exp{\cdot}$ denote the total expectation. For convenience, we define the following shorthands
\[
    \begin{dcases}
        \Delta_f& =f(x_0)-\inf f, \\
        \Delta_g& =g(x_0)-\inf g, \\
        \nu_{f,k}^2 &= \frac{\sigma_f^2}{B_{f,k}}, \\
        \nu_{g,k}^2 &= \frac{\sigma_g^2}{B_{g,k}}, \\
        \tau_k^2 &= 2\nu_{f,k}^2+2\left(\frac{3\beta_k}{2}+\frac{2G_f}{\sqrt {\rho_k}}\right)^2\nu_{g,k}^2, \\
        C_{f,k} &= G_f\tau_k+L_f\eta_k\tau_k^2, \\
        C_{g,k} &= \beta_k \rho_k+\frac{G_f}{2}\sqrt{\rho_k}+\frac{\tau_k^2}{2\beta_k}+L_g\eta_k \tau_k^2, \\
        R_k &= 4C_{f,k}+8\beta_k C_{g,k}+\frac{C_{g,k}}{L_g\eta_k}+8\beta_k L_g\eta_k G_f^2, \\
        \delta_{f, k}&= f(x_k)-\ExpD{k}{f(x_{k+1})}, \\
        \delta_{g, k}&= g(x_k)-\ExpD{k}{g(x_{k+1})}.
    \end{dcases}
\]

\subsection{Preliminary Results}

\begin{lemma}[{Quadratic Bound \cite{cao_complexity_2025}}]\label{lem:quad}
	For $D\ge 0$, if $x\ge 0$ and $x\le C+D\sqrt{x}$, then $x\le 2C+D^2$. 
\end{lemma}
\begin{proof}
    This follows from $D\sqrt x\le \frac{D^2}{2}+\frac{x}{2}$.
\end{proof}

\begin{lemma}[Multiplier Properties]\label{lem:helpful}
    Let $a, b\in \H$ and $\beta, \rho>0$. Define 
    \[
        \lambda=\frac{[\beta\|b\|^2-\langle a, b\rangle]_+}{\|b\|^2+\rho}\quad \text{and}\quad d=a+\lambda b.
    \]
    Then it holds that 
    \begin{subequations}
    \begin{align}
        &\langle b, d\rangle \ge \beta\|b\|^2-\rho\lambda, \label{lem:helpful:1a}\\
        &\lambda \langle b, d\rangle =\beta \lambda \|b\|^2-\rho\lambda^2, \label{lem:helpful:1b}\\
        &0\le \lambda \le \beta+\frac{\|a\|}{2\sqrt{\rho}}, \label{lem:helpful:1c} \\
       &\lambda \|b\|^2\le \beta\|b\|^2+\|a\|\cdot \|b\|. \label{lem:helpful:1d}
    \end{align}
    \end{subequations}
\end{lemma}
\begin{proof}
    We consider two cases:
    \begin{itemize}
        \item If $\beta\|b\|^2-\langle a, b\rangle\le 0$, then $\lambda=0$ and $d=a$. As such, \eqref{lem:helpful:1a}-\eqref{lem:helpful:1d} are readily verified.
        \item If $\beta\|b\|^2-\langle a, b\rangle> 0$, then 
        \[
            \lambda=\frac{\beta\|b\|^2-\langle a, b\rangle}{\|b\|^2+\rho}.
        \]
        As $d=a+\lambda b$, Inequality \eqref{lem:helpful:1a} holds with equality. Multiplying by $\lambda$ yields Equality \eqref{lem:helpful:1b}. Moreover, it holds that 
        \[
            \lambda\le \frac{\beta\|b\|^2+\|a\|\cdot \|b\|}{\|b\|^2+\rho}\le \beta+\frac{\|a\|\cdot \|b\|}{\|b\|^2+\rho}\le \beta+\frac{\|a\|}{2\sqrt{\rho}},
        \]
        where the latter follows by $\frac{t}{t^2+\rho}\le \frac{1}{2\sqrt{\rho}}$ for all $t\ge 0$. As such, Inequality \eqref{lem:helpful:1c} holds. Finally, 
        \[
            \lambda \|b\|^2=\left(\beta\|b\|^2-\langle a, b\rangle\right)\cdot \frac{\|b\|^2}{\|b\|^2+\rho}\le \beta \|b\|^2+\|a\|\cdot \|b\|,
        \]
        thus yielding Inequality \eqref{lem:helpful:1d}.
    \end{itemize}
\end{proof}

\begin{lemma}[Stability of Stochastic Direction]\label{lem:stability}
    Define $D_{\beta, \rho}\colon \H\times \H\to \H$ by
    \[
        D_{\beta, \rho}(a, b)=a+\frac{[\beta \|b\|^2-\langle a, b\rangle]_+}{\|b\|^2+\rho}b.
    \]
    Then, for any $a,a',b,b'\in \H$, it holds that
    \[
        \|D_{\beta, \rho}(a, b)-D_{\beta, \rho}(a', b')\|\le \|a-a'\|+\left(\frac{3\beta}{2}+\frac{2\|a\|}{\sqrt{\rho}}\right)\|b-b'\|.
    \]
    Specifically, almost surely,
    \[
        \ExpD{k}{\|d_k-\hat d_k\|^2}\le \tau^2_k.
    \]
\end{lemma}
\begin{proof}
    The following hold:
    \begin{itemize}
        \item For a fixed $b\in \H$, we define $\phi_b(a)=a+\frac{[\beta\|b\|^2-\langle a, b\rangle]+}{\|b\|^2+\rho}b$. On the half-space $\beta\|b\|^2-\langle a, b\rangle\le 0$, $\phi_b\equiv I$, and therefore $D\phi_b(a)\equiv I$ on the interior of the half-space, where it then holds that $\|D\phi_b(a)\|_{\text{op}}=1$. On the interior of the other half-space, namely for $\beta\|b\|^2-\langle a, b\rangle> 0$, it holds that $D\phi_b(a)\equiv I-\frac{b\otimes b}{\|b\|^2+\rho}$, whose eigenvalues are $1$ on $b^\perp$ and $\frac{\rho}{\|b\|^2+\rho}\in [0, 1)$ for $\rho\ge 0$. Specifically, $\|D\phi_b(a)\|_{\text{op}}\le 1$. 

        Now, fix $a$ and $a'$. If both lie within the same half-space, then, it holds that
        \begin{align*}
            \|\phi_b(a)-\phi_b(a')\|
            &\le \left|\int_0^1 D\phi_b(a+t(a'-a))[a'-a]dt\right|\\
            &\le\int_0^1 \|D\phi_b(a+t(a'-a))\|_{\text{op}}\|a'-a\|dt \le \|a-a'\|.
        \end{align*}
        If they do not lie within the same half-space, define $\tilde a$ as the unique point on the segment $[a, a']$ satisfying $\beta\|b\|^2-\langle a, b\rangle= 0$. Then it holds that
        \begin{align*}
            \|\phi_b(a)-\phi_b(a')\|
            &\le \|\phi_b(a)-\phi_b(\tilde a)\|+\|\phi_b(\tilde a)-\phi_b(a')\|
            &\le \|a-\tilde a\|+\|\tilde a-a'\|=\|a-a'\|,
        \end{align*}
        where the latter holds as $a, \tilde a, a'$ are colinear. Specifically, it holds that, for fixed $b\in \H$ and arbitrary $a, a'\in \H$, 
        \[
            \|D_{\beta, \rho}(a, b)-D_{\beta, \rho}(a', b)\|\le \|a-a'\|.
        \]
        \item Now fix $a\in \H$, and define $\lambda_a(b)=\frac{\beta\|b\|^2-\langle a, b\rangle}{\|b\|^2+\rho}$ and $\psi_a(b)=a+[\lambda_a(b)]_+b$. In the inactive region, $\psi_a(b)\equiv a$, so that $D\psi_a(b)\equiv 0$. In the active region, it holds that 
        \[
            D\psi_a(b)[u]=D\lambda_a(b)[u]b+\lambda_a(b)u,
        \]
        where
        \[
            D\lambda_a(b)[u]=\frac{2\rho\beta\langle b, u\rangle}{(\|b\|^2+\rho)^2}-\frac{\langle a, u\rangle}{\|b\|^2+\rho}+\frac{2\langle b, u\rangle \langle a, b\rangle}{(\|b\|^2+\rho)^2}.
        \]
        As such, 
        \[
            \|D\psi_a(b)[u]\|\le \left(\frac{2\rho\beta\|b\|^2}{(\|b\|^2+\rho)^2}+\frac{\|a\| \|b\|}{\|b\|^2+\rho}+\frac{2 \|b\|^3 \|a\|}{(\|b\|^2+\rho)^2}+\lambda_a(b)\right)\|u\|.
        \]
        Since it holds that 
        \[
            \frac{\|b\|^2}{(\|b\|^2+\rho)^2}\le \frac{1}{4\rho}, \quad \frac{\|b\|}{\|b\|^2+\rho}\le \frac{1}{2\sqrt{\rho}}, \quad \frac{\|b\|^3}{(\|b\|^2+\rho)^2}\le \frac{1}{2\sqrt{\rho}},
        \]
        and, by Inequality \eqref{lem:helpful:1c} of Lemma \ref{lem:helpful}, 
        \[
            \lambda_a(b)\le \beta+\frac{\|a\|}{2\sqrt{\rho}},
        \]
        we deduce that 
        \[
            \|D\psi_a(b)\|_{\text{op}}\le \frac{\beta}{2}+\frac{\|a\|}{2\sqrt \rho}+\frac{\|a\|}{\sqrt \rho}+\beta+\frac{\|a\|}{2\sqrt \rho}=\frac{3\beta}{2}+\frac{2\|a\|}{\sqrt \rho}.
        \]
        We note that the inactive and active regions are no longer half-spaces. However, for a given $a\in \H$, any segment $[b, b']$ can be split into finitely many segments that each lie entirely within one of the regions, since these regions are parametrized by a polynomial. As such, by the same reasoning as the previous case, it holds that, for fixed $a\in \H$ and arbitrary $b, b'\in \H$, 
        \[
            \|D_{\beta, \rho}(a, b)-D_{\beta, \rho}(a, b')\|\le \left(\frac{3\beta}{2}+\frac{2\|a\|}{\sqrt \rho}\right)\|b-b'\|.
        \]
    \end{itemize}
    Since
    \[
        \|D_{\beta, \rho}(a, b)-D_{\beta,\rho}(a', b')\|\le \|D_{\beta, \rho}(a, b)-D_{\beta,\rho}(a, b')\|+\|D_{\beta, \rho}(a', b')-D_{\beta,\rho}(a, b')\|,
    \]
    the above two bounds yield the first result.
    
    The second result follows by using $d_k=D_{\beta_k, \rho_k}( \nabla f(x_k),  \nabla g(x_k))$ and $\hat d_k=D_{\beta_k, \rho_k}(\nablafk, \nablagk)$, and recalling the bounds 
    \[
    	\ExpD{k}{\|\nabla f(x_k)-\nablafk\|^2}\le \nu_{f,k}^2, \quad \ExpD{k}{\|\nabla g(x_k)-\nablagk\|^2}\le \nu_{g,k}^2, \quad \|\nabla f(x_k)\|\le G_f.
    \]
    As such, 
    {\small\begin{align*}
        \ExpD{k}{\|d_k-\hat d_k\|^2}&\le 2\ExpD{k}{\|\nabla f(x_k)-\nablafk\|^2}+2\left(\frac{3\beta_k}{2}+\frac{2\|\nabla f(x_k)\|}{\sqrt \rho_k}\right)^2\ExpD{k}{\|\nabla g(x_k)-\nablagk\|^2} \\
        &\le 2\nu_{f, k}^2+2\left(\frac{3\beta_k}{2}+\frac{2G_f}{\sqrt {\rho_k}}\right)^2\nu_{g, k} = \tau_k^2,
    \end{align*}
    }as wanted.
\end{proof}

\subsection{Descent Analysis and Energy Bound}

We now derive descent lemmas for both the upper- and lower-level objectives. Combining these bounds yields a weighted stationarity bound, from which the convergence rates in Section \ref{app:proof_coro_rates} follow. 

\begin{lemma}[Upper-Level Descent]\label{lem:upper}
    For every $k\ge 0$, almost surely,
    \[
        (1-L_f\eta_k)\|d_k\|^2\le \frac{\delta_{f, k}}{\eta_k}+\beta_k^2\|\nabla g(x_k)\|^2+\beta_k G_f\|\nabla g(x_k)\|+C_{f,k}.
    \]
\end{lemma}
\begin{proof}
    By $L_f$-smoothness of $f$ and Lemma \ref{lem:stability},
    {\small\begin{align*}
        \ExpD{k}{f(x_{k+1})}
        &=\ExpD{k}{f(x_k-\eta_k \hat d_k)} \\
        &\le \ExpD{k}{f(x_k) -\eta_k\langle \nabla f(x_k), \hat d_k\rangle+\frac{L_f\eta_k^2}{2}\|\hat d_k\|^2} \\
        &= \ExpD{k}{f(x_k) -\eta_k\langle \nabla f(x_k), \hat d_k-d_k\rangle-\eta_k\langle \nabla f(x_k), d_k\rangle+\frac{L_f\eta_k^2}{2}\|\hat d_k-d_k+d_k\|^2} \\
        &\le \ExpD{k}{f(x_k) +\eta_k \|\nabla f(x_k)\|\cdot  \|\hat d_k-d_k\|-\eta_k\langle \nabla f(x_k), d_k\rangle+L_f\eta_k^2\|\hat d_k-d_k\|^2+L_f\eta_k^2\|d_k\|^2} \\
        &\le f(x_k) +\eta_k G_f\tau_k-\eta_k\langle \nabla f(x_k), d_k\rangle+L_f\eta_k^2\tau_k^2+L_f\eta_k^2\|d_k\|^2.
    \end{align*}
    }The result follows by Lemma \ref{lem:helpful}, by noting that
    \begin{align*}
        \langle \nabla f(x_k), d_k\rangle
        &=\langle d_k-\lambda_k \nabla g(x_k), d_k\rangle \\
        &=\|d_k\|^2-\lambda_k \langle \nabla g(x_k), d_k\rangle \\
        &\stackrel{\eqref{lem:helpful:1b}}{=} \|d_k\|^2-\beta_k\lambda_k \|\nabla g(x_k)\|^2+\rho_k\lambda_k^2 \\
        &\stackrel{\eqref{lem:helpful:1d}}{\ge} \|d_k\|^2-\beta_k^2 \|\nabla g(x_k)\|^2-\beta_k\|\nabla f(x_k)\|\cdot \|\nabla g(x_k)\|+\rho_k\lambda_k^2,
    \end{align*}
    and by bounding away $\rho_k\lambda_k^2\ge 0$.
\end{proof}

\begin{lemma}[Lower-Level Descent]\label{lem:lower}
    For every $k\ge 0$, almost surely,
    \[
        \frac{\beta_k}{2}\|\nabla g(x_k)\|^2\le \frac{\delta_{g, k}}{\eta_k}+L_g\eta_k\|d_k\|^2+C_{g,k}.
    \]
\end{lemma}
\begin{proof}
    By $L_g$-smoothness of $g$, mirroring the steps in the proof of Lemma \ref{lem:upper}, 
    \begin{align*}
        \ExpD{k}{g(x_{k+1})}
        &\le g(x_k) +\eta_k \|\nabla g(x_k)\|\tau_k-\eta_k\langle \nabla g(x_k), d_k\rangle+\eta_k^2L_g\tau_k^2+\eta_k^2L_g\|d_k\|^2.
    \end{align*}
    By Young's Inequality, it holds that 
    \[
        \|\nabla g(x_k)\|\tau_k\le \frac{\beta_k\|\nabla g(x_k)\|^2}{2}+\frac{\tau_k^2}{2\beta_k},
    \]
    and by Lemma \ref{lem:helpful}, we get that 
    \[
        \langle \nabla g(x_k), d_k\rangle \stackrel{\eqref{lem:helpful:1a}}\ge \beta_k \|\nabla g(x_k)\|^2-\rho_k\lambda_k\stackrel{\eqref{lem:helpful:1c}}\ge \beta_k \|\nabla g(x_k)\|^2-\beta_k \rho_k - \rho_k \frac{G_f}{2\sqrt {\rho_k}},
    \]
    which concludes.
\end{proof}
\begin{lemma}[Combined Descent]\label{lem:combined}
    Assume $\eta_k\le \frac{1}{2(L_f+2L_g)}$ and $0<\beta_k\le 1$. For every $k\ge 0$, almost surely,
    \[
        \|d_k\|^2\le \frac{4}{\eta_k}\delta_{f, k}+\frac{8\beta_k\eta_kL_g+1}{\eta_k^2L_g}\delta_{g, k}+R_{k},
    \]
    and
    \[
        \frac{\beta_k}{2}\|\nabla g(x_k)\|^2\le 4L_g\delta_{f, k}+\frac{8L_g\beta_k\eta_k+2}{\eta_k}\delta_{g, k}+\eta_kL_gR_k+C_{g,k}.
    \]
\end{lemma}
\begin{proof}
    Define $H_k=\frac{\delta_{g, k}}{\eta_k}+L_g\eta_k\|d_k\|^2+C_{g,k}$, so that $H_k\ge \frac{\beta_k}{2}\|\nabla g(x_k)\|^2$ by Lemma \ref{lem:lower}. Then, since $\beta_k^2\|\nabla g(x_k)\|^2\le 2\beta_k H_k$ and $\beta_k G_f\|\nabla g(x_k)\|\le \sqrt{2\beta_k}G_f\sqrt{H_k}$, we deduce from Lemma \ref{lem:upper} that
    \begin{align*}
    	(1-L_f\eta_k)\|d_k\|^2
	& \le \frac{\delta_{f, k}}{\eta_k}+2\beta_k H_k+\sqrt{2\beta_k}G_f\sqrt{H_k}+C_{f,k} \\
	& = \frac{\delta_{f, k}}{\eta_k}+\frac{2\beta_k\delta_{g, k}}{\eta_k}+2\beta_k L_g\eta_k\|d_k\|^2+\sqrt{2\beta_k}G_f\sqrt{H_k}+C_{f,k}+2\beta_k C_{g,k}.
    \end{align*}
    As $1-L_f\eta_k-2\beta_k L_g\eta_k\ge \frac12$, we get
    \[
    	\|d_k\|^2 \le \frac{2\delta_{f, k}}{\eta_k}+\frac{4\beta_k\delta_{g, k}}{\eta_k}+2\sqrt{2\beta_k}G_f\sqrt{H_k}+2C_{f,k}+4\beta_k C_{g,k},
    \]
    which may be rewritten as
    \[
    	\frac{H_k}{L_g\eta_k} \le \frac{2\delta_{f, k}}{\eta_k}+\frac{4\beta_k\delta_{g, k}}{\eta_k}+\frac{\delta_{g, k}}{L_g\eta_k^2}+2\sqrt{2\beta_k}G_f\sqrt{L_g\eta_k}\sqrt{\frac{H_k}{L_g\eta_k}}+2C_{f,k}+4\beta_k C_{g,k}+\frac{C_{g,k}}{L_g\eta_k}.
    \]
    Using Lemma \ref{lem:quad} with $x=\frac{H_k}{L_g\eta_k}$, we get that 
    \[
    	\frac{H_k}{L_g\eta_k}\le 	\frac{4\delta_{f, k}}{\eta_k}+\frac{8\beta_k\delta_{g, k}}{\eta_k}+\frac{2\delta_{g, k}}{L_g\eta_k^2}+4C_{f,k}+8\beta_k C_{g,k}+\frac{2C_{g,k}}{L_g\eta_k}+8\beta_k G_f^2L_g\eta_k,
    \]
    which, equivalently, may be written as the first wanted inequality. Substituting this into Lemma \ref{lem:lower} yields the second claimed inequality.
\end{proof}

\begin{theorem}[Residual Bounds]\label{thm:main}
    Assume $\eta_k\le \frac{1}{2(L_f+2L_g)}$, $0<\beta_k\le 1$, $(\eta_k)$ and $(\beta_k)$ are nonincreasing, and $N_K$ is sampled according to \eqref{eq:def_P}. Then it holds that, for $K\ge 1$,
    \[
        \Exp{\|\nabla g(x_{N_K})\|^2}\le \frac{8\eta_0L_g}{\sum_{k=0}^{K-1}\eta_k\beta_k}\Delta_{f}+\frac{16\eta_0\beta_0L_g+4}{\sum_{k=0}^{K-1}\eta_k\beta_k}\Delta_g+\frac{2\sum_{k=0}^{K-1}\left(\eta_k^2L_gR_{k}+\eta_kC_{g,k}\right)}{{\sum_{k=0}^{K-1}\eta_k\beta_k}}.
    \]
    Moreover, if $(\beta_k/\eta_k)$ is nonincreasing,
    \[
         \Exp{\|d_{N_K}\|^2}\le \frac{4\beta_0}{\sum_{k=0}^{K-1}\eta_k\beta_k}\Delta_{f}+\frac{8\beta_0^2L_g+\frac{\beta_0}{\eta_0}}{L_g\sum_{k=0}^{K-1}\eta_k\beta_k}\Delta_g+\frac{\sum_{k=0}^{K-1}\eta_k\beta_kR_k}{\sum_{k=0}^{K-1}\eta_k\beta_k},
    \]
    and, if $(\beta_k/\eta_k)$ is nondecreasing,
    \[
        \Exp{\|d_{N_K}\|^2}\le \frac{\beta_{K-1}}{\eta_{K-1}}\left(\frac{4\eta_0}{\sum_{k=0}^{K-1}\eta_k\beta_k}\Delta_f+\frac{8L_g\beta_0\eta_0+1}{L_g{\sum_{k=0}^{K-1}\eta_k\beta_k}}\Delta_g+\frac{\sum_{k=0}^{K-1}\eta_k^2R_{k}}{\sum_{k=0}^{K-1}\eta_k\beta_k}\right).
    \]
\end{theorem}
\begin{proof}
    Note that
    \[
        \sum_{k=0}^{K-1}\Exp{\delta_{f, k}}=f(x_0)-\Exp{f(x_K)}\le f(x_0)-\inf f=\Delta_f,
    \]
    and, similarly, $\sum_{k=0}^{K-1}\Exp{\delta_{g, k}}\le \Delta_g$. Moreover, for any nonincreasing sequence $(w_k)$, summation by parts allows us to write 
    {\small\begin{align*}
        \sum_{k=0}^{K-1}w_k\Exp{\delta_{f, k}}
        &=(w_0(f(x_0)-\inf f)-w_K\Exp{f(x_K)-\inf f})+\sum_{k=0}^{K-1}\Exp{f(x_{k+1})-\inf f}(w_{k+1}-w_k) \\
        &\le w_0(f(x_0)-\inf f),
    \end{align*}}which we shall use throughout this proof. A similar result of course holds for $g$.
    
    We separate the proof of the three inequalities:
    \begin{enumerate}
        \item For the first inequality, we multiply the second inequality of Lemma \ref{lem:combined} by $2\eta_k$, take total expectations, and sum for $k=0, \ldots, K-1$, to obtain 
        {\small\[
            \sum_{k=0}^{K-1}\eta_k\beta_k\Exp{\|\nabla g(x_{k})\|^2}\le 8L_g\sum_{k=0}^{K-1}\eta_k\Exp{\delta_{f,k}}+\sum_{k=0}^{K-1}(16\eta_k\beta_kL_g+4)\Exp{\delta_{g, k}}+2\sum_{k=0}^{K-1}\eta_k^2L_gR_{k}+\eta_kC_{g,k}.
        \]
        }As $(\eta_k)$ and $(\eta_k\beta_k)$ are nonincreasing, we can apply summation by parts to obtain the wanted inequality, after division by $\sum_{k=0}^{K-1}\eta_k\beta_k$.
        \item By multiplying the first inequality of Lemma \ref{lem:combined} by $\eta_k\beta_k$, we obtain
        \[
            \sum_{k=0}^{K-1}\eta_k\beta_k\Exp{\|d_{k}\|^2}\le 4\sum_{k=0}^{K-1}\beta_k\Exp{\delta_{f, k}}+\sum_{k=0}^{K-1}\frac{8\beta_k^2L_g+\frac{\beta_k}{\eta_k}}{L_g}\Exp{\delta_{g,k}}+\sum_{k=0}^{K-1}\eta_k\beta_kR_k.
        \]
        As $(\beta_k)$ and $(\beta_k/\eta_k)$ are nonincreasing, we can apply summation by parts to obtain the wanted inequality, after division by $\sum_{k=0}^{K-1}\eta_k\beta_k$.
        \item Finally, we multiply the first inequality of Lemma \ref{lem:combined} by $\eta_k^2$, we obtain for $k\le K-1$, since $(\beta_k/\eta_k)$ is nondecreasing,
        \begin{align*}
            \eta_k\beta_k\|d_k\|^2
            &\le \frac{\beta_k}{\eta_k}\eta_k^2\|d_k\|^2 \\
            &\le \frac{\beta_{K-1}}{\eta_{K-1}}\left(4\eta_k\delta_{f, k}+\frac{8L_g\beta_k\eta_k+1}{L_g}\delta_{g, k}+\eta_k^2R_{k}\right).
        \end{align*}
        As such, taking total expectation 
        and summing for $k=0,\ldots, K-1$ yields
        \[
            \sum_{k=0}^{K-1}\eta_k\beta_k\Exp{\|d_k\|^2}\le \frac{\beta_{K-1}}{\eta_{K-1}}\left(4\sum_{k=0}^{K-1}\eta_k\Exp{\delta_{f, k}}+\sum_{k=0}^{K-1}\frac{8L_g\beta_k\eta_k+1}{L_g}\Exp{\delta_{g, k}}+\sum_{k=0}^{K-1}\eta_k^2R_{k}\right).
        \]
        As $(\eta_k)$ and $(\eta_k\beta_k)$ are nonincreasing, we may apply summation by parts to obtain the wanted quantity, after division by $\sum_{k=0}^{K-1}\eta_k\beta_k$.
    \end{enumerate}
\end{proof}

\subsection{Proof of Corollary \ref{coro:rates}}\label{app:proof_coro_rates}

\subsubsection{Horizon-Dependent Parameters}

We consider the constant parameter setting of Theorem \ref{thm:main}, namely $\eta_k\equiv \eta, \beta_k\equiv \beta, \rho_k\equiv \rho, B_{f,k}\equiv B_f, B_{g,k}\equiv B_g$. This implies that $\tau_k\equiv \tau, C_{f,k}\equiv C_f, C_{g,k}\equiv C_g, R_k\equiv R$.

We assume that $\eta, \beta, \rho\le 1$ and that $B_f, B_g\ge 1$.
Therefore,
\[
    C_f=\mathcal O\left(\tau+\eta\tau^2\right)\quad \text{and}\quad C_g=\mathcal O\left(\beta\rho+\sqrt \rho+\frac{\tau^2}{\beta}+\eta\tau^2\right)=\mathcal O\left(\sqrt \rho+\frac{\tau^2}{\beta}\right),
\]
so that 
\begin{align*}
    R=\mathcal O\left(C_f+\beta C_g+\frac{C_g}{\eta}+\beta\eta\right) &=\mathcal O\left(\tau+\eta\tau^2+\beta \sqrt{\rho}+\tau^2+\frac{\sqrt{\rho}}{\eta}+\frac{\tau^2}{\beta\eta}+\beta\eta\right) \\
    &= \mathcal O\left(\tau + \frac{\tau^2}{\beta\eta}+\frac{\sqrt{\rho}}{\eta}+\beta\eta\right)
\end{align*}
Theorem \ref{thm:main} thus yields
\begin{align*}
    \Exp{\|d_{N_K}\|^2}
    &\le \mathcal O\left(\frac{1}{\eta K}+\frac{\beta}{\eta K}+\frac{1}{\eta^2K}+R\right) \\
    &\le \mathcal O\left(\frac{1}{\eta K}+\frac{\beta}{\eta K}+\frac{1}{\eta^2K}+\tau + \frac{\tau^2}{\beta\eta}+\frac{\sqrt{\rho}}{\eta}+\beta\eta \right) \\
    &\le \mathcal O\left(\frac{1}{\eta^2K}+\tau + \frac{\tau^2}{\beta\eta}+\frac{\sqrt{\rho}}{\eta}+\beta\eta \right).
\end{align*}
Balancing $\tau$ yields $\tau\propto \beta\eta$. Balancing $\rho$ yields $\rho\propto \beta^2\eta^4$. 

Finally, balancing $\beta$ yields $\beta\propto \eta^{-3}K^{-1}$. In order to achieve $\tau\propto \beta\eta\propto \eta^{-2}K^{-1}$, we must have 
\[
    \frac{1}{B_f}+\frac{\beta^2}{B_g}+\frac{1}{\rho B_g}\propto \eta^{-4}K^{-2}\implies B_f\propto \eta^4K^2, \quad B_g\propto \max(\eta^{-2}, \eta^6K^4).
\]
In that case, it holds that 
\[
    \Exp{\|d_{N_K}\|^2}\le\mathcal O(\eta^{-2}K^{-1}).
\]
Moreover, Theorem \ref{thm:main} also yields that
\begin{align*}
    \Exp{\|\nabla g(x_{N_K})\|^2}
    &\le \mathcal O\left(\frac{1}{\beta K}+\frac{1}{K}+\frac{1}{\eta\beta K}+\frac{\eta R}{\beta}+\frac{C_g}{\beta}\right) \\
    &\le \mathcal O\left(\frac{1}{\beta K}+\frac{1}{K}+\frac{1}{\eta\beta K}+\frac{\eta\tau}{\beta}+\frac{\tau^2}{\beta^2}+\frac{\sqrt{\rho}}{\beta}+\eta^2+\frac{\sqrt{\rho}}{\beta}+\frac{\tau^2}{\beta^2}\right) \\
    &= \mathcal O\left(\frac{1}{\eta\beta K}+\frac{\eta\tau}{\beta}+\frac{\tau^2}{\beta^2}+\frac{\sqrt{\rho}}{\beta}+\eta^2\right).
\end{align*}
Under the chosen parameters, the remaining terms are of order $\mathcal O(\eta^2)$ or lower, so that
\[
    \Exp{\|\nabla g(x_k)\|^2}\le \mathcal O\left(\eta^2\right).
\]
Setting $\eta\propto K^{-a}$ for $a\in (0, 1/3)$ yields the wanted result.

\subsubsection{Anytime Parameters}

The parameter choice follows from the parameters selected in the finite-horizon setting. Under these choices, it holds that 
\[
    \beta_k\propto (k+1)^{3a-1}, \quad \rho_k\propto (k+1)^{2a-2},
\]
\[
    B_{f, k}\propto (k+1)^{2-4a}, \quad B_{g, k}\propto (k+1)^{\max(2a, 4-6a)}=(k+1)^{4-6a}.
\]
Specifically, 
\begin{align*}
    \tau_k^2
    &=\mathcal O\left(\frac{1}{B_{f,k}}+\frac{\beta_k^2}{B_{g, k}}+\frac{1}{\rho_kB_{g, k}}\right) \\
    &=\mathcal O\left((k+1)^{4a-2}+(k+1)^{6a-2-(4-6a)}+(k+1)^{2-2a-(4-6a)}\right) \\
    &=\mathcal O((k+1)^{\max(12a-6,4a-2)}) \\
    &=\mathcal O((k+1)^{4a-2})
\end{align*}
Thus 
\[
    C_{f,k}\propto \tau_k+\eta_k\tau_k^2\propto (k+1)^{2a-1}+(k+1)^{3a-2}\propto (k+1)^{2a-1},
\]
and 
\begin{align*}
    C_{g, k}
    &\propto \beta_k\rho_k+\sqrt{\rho_k}+\frac{\tau_k^2}{\beta_k}+\eta_k\tau_k^2 \\
    &\propto (k+1)^{5a-3}+(k+1)^{a-1}+(k+1)^{a-1}+(k+1)^{3a-2} \\
    &\propto (k+1)^{a-1}.
\end{align*}
Therefore
\begin{align*}
    R_k
    &\propto C_{f, k}+\beta_kC_{g, k}+\frac{C_{g, k}}{\eta_k}+\beta_k\eta_k \\
    &\propto (k+1)^{2a-1}+(k+1)^{4a-2}+(k+1)^{2a-1}+(k+1)^{2a-1} \\
    &\propto (k+1)^{2a-1},
\end{align*}
and
\[
    R_k\eta_k\beta_k
    \propto (k+1)^{4a-2} \quad \text{and}\quad 
    R_{k}\eta_k^2\propto (k+1)^{-1}.
\]
As such, Theorem \ref{thm:main} gives us that, if $(\beta_k/\eta_k)\propto (k+1)^{4a-1}$ is nonincreasing, namely if $a\le 1/4$, then
\[
    \Exp{\|d_{N_K}\|^2}\le \mathcal O\left(\frac{1+\sum_{k=0}^{K-1}(k+1)^{4a-2}}{\sum_{k=0}^{K-1}(k+1)^{2a-1}}\right)=\begin{dcases}
        \mathcal O(K^{-2a})\quad &\text{if $0<a<1/4$}, \\
        \mathcal O(K^{-1/2}\log(K)) &\text{if $a=1/4$},
    \end{dcases},
\]
and if $(\beta_k/\eta_k)\propto (k+1)^{4a-1}$ is nondecreasing, namely if $a\ge 1/4$, then
\[
    \Exp{\|d_{N_K}\|^2}\le \mathcal O\left(K^{4a-1}\frac{1+\sum_{k=0}^{K-1}(k+1)^{-1}}{\sum_{k=0}^{K-1}(k+1)^{2a-1}}\right)=\mathcal O(K^{2a-1}\log(K)),
\]
which combined yields 
\[
    \Exp{\|d_{N_K}\|^2}\le \begin{dcases}
        \mathcal O(K^{-2a})\quad &\text{if $0<a<1/4$}, \\
        \mathcal O(K^{2a-1}\log(K))&\text{if $1/4\le a<1/3$}.
    \end{dcases}.
\]
We also have
\[
    \Exp{\|\nabla g(x_{N_K})\|^2}\le \mathcal O\left(\frac{1+\sum_{k=0}^{K-1}(k+1)^{-1}}{\sum_{k=0}^{K-1}(k+1)^{2a-1}}\right)=\mathcal O(K^{-2a}\log(K)).
\]
\section{Experimental Setup}\label{app:XP}

The goal of the numerical example in Section \ref{sec:num} is to provide a toy example showcasing a comparison between the methods SDBGD, SDBPG, PR-SDBPG and VR-PR-SDBPG. In this section, we provide a complete description of the experiment, allowing for complete reproducability. 

\textbf{Problem Setup.} As mentioning in Section \ref{sec:num}, we select
\[
    f(x)=\sqrt{1+\|x-c\|^2} \quad \text{and}\quad g(x_1,x_2)=2-\cos(x_1)-\exp(-x_2^2/2),
\]
where $c=(7.1, 1)$. This choices implies that 
\[
    \argmin g=\{(2\pi m, 0)\colon m\in \mathbb Z\},
\]
so that the solution to \eqref{eq:SSNBP} is $(2\pi,0)$. These functions satisfy our assumptions with $L_f=1, L_g=1, G_f=1$. 

\textbf{Stochastic Oracle.} We having access to stochastic gradients given by 
\[
    \nabla f_\xi(x)=\nabla f(x)+\xi\quad\text{and}\quad \nabla g_\zeta(x)=\nabla g(x)+\zeta,
\]
where $\xi\sim \mathcal N(0, \sigma_f^2I)$ and $\zeta\sim \mathcal N(0, \sigma_g^2I)$, so that assumption of uniformly bounded variance is satisfied. We select $\sigma_f=\sigma_g=0.5$.

\textbf{Algorithmic Setup.} All methods are implemented in an anytimes fashion. Our method, SDBGD, uses the anytime schedule described in Corollary \ref{coro:rates} with $a=1/4$. For the methods in Ahmadi et al. \cite{ahmadi_stochastic_2026}, whose theory only cover horizon-dependent schedules, we form anytime schedules by replacing the horizon $K$ by the current effective horizon $k+1$. Each method is started at the same initial point $x_0=(1.5, 1.5)$ and is given the same stochastic oracle budget $\mathcal B=10^8$.

\textbf{Parameter Selection.} The choice of parameters is given in Table \ref{tab:parameters}. The powers of $t_k=(k+1)$ follow the theoretical results, whereas the proportionality constants are chosen manually and not varied between runs. 

\vspace{-1em}
\begin{table}[H]
    \centering
    \small
    \caption{Parameter choices for the compared methods. Here, $t_k=k+1$ for convenience.}
    \begin{tabular}{|c|c|c|c|c|c|c|} \hline
        \shortstack{\textbf{Method}\strut\\\strut} & \shortstack{\textbf{Step-Size}\strut \\ ($\eta_k$)\strut} & \shortstack{\textbf{Barrier}\strut \\ ($\beta_k$)\strut} & \shortstack{\textbf{Regul.}\strut \\ ($\rho_k, \gamma_k$)\strut} & \shortstack{\textbf{Upp.-Bat.}\strut \\ ($B_{f, k}$)\strut} & \shortstack{\textbf{Low.-Bat.}\strut \\ ($B_{g, k}$)\strut} & \shortstack{\textbf{Other}\strut\\\strut}  \\ \hline
        SDBGD & $0.05t_k^{-1/4}$ & $0.5t_k^{-1/4}$ & $t_k^{-3/2}$ & $\lfloor t_k\rfloor$ & $\lfloor t_k^{5/2}\rfloor$ & N/A \\ \hline
        SDBPG & $0.05t_k^{-1/4}$ & $0.5t_k^{-1/4}$ & $t_k^{-1}$ & $\lfloor t_k\rfloor$ & $\lfloor t_k^{2}\rfloor$ & N/A \\ \hline
        PR-SDBPG & $0.05t_k^{-1/2}$ & $0.5t_k^{1/4}$ & $1$ & $\lfloor t_k\rfloor$ & $\lfloor t_k^{3/2}\rfloor$ & $\mu_k=1$ \\ \hline
        \shortstack{VR-PR-\strut\\SDBPG\strut} & $0.05t_k^{-1/2}$ & $0.5t_k^{1/4}$ & $1$ & $\lfloor t_k^{1/2}\rfloor$ & $\lfloor t_k\rfloor$ & \shortstack{$\mu_k=1$\strut\\$\alpha_k=0.2t_k^{-1/2}$\strut} \\ \hline
    \end{tabular}
    \label{tab:parameters}
\end{table}
\vspace{-1em}

\textbf{Oracle Accounting.} During iteration $k$, the methods SDBGD, SDBPG and PR-SDBPG use $B_{f,k}+B_{g,k}$ stochastic oracle calls. Method VR-PR-SDBPG makes use of $2(B_{f,k}+B_{g,k})$. An iteration is executed only if its complete oracle cost fits within the remaining oracle budget. 

\textbf{Evaluation Protocol.} We perform $10$ independent runs, and plot the empirical mean of the upper-level stationarity measure $\|d_k\|^2$ and if the lower-level stationarity measure $\|\nabla g(x_k)\|^2$. We note that full gradients are only used for evaluation purposes and are never called by the algorithm.

\textbf{Computer Resources.} The experiment is written in Python 3.13, and was executed on an Apple Silicon MacBook Pro, with a M5 chip and 16GB of RAM.

\textbf{Code Availability.} The code is available on the author's GitHub page: \url{https://github.com/DanielCortild/Stochastic-Nonconvex-Bilevel}.

\end{document}